\documentclass[12pt]{amsart}

\usepackage{amsmath,amssymb,amsthm}
\usepackage[dvipdfm]{graphicx}
\usepackage[dvips]{color}
\usepackage{version}

\newtheorem{Thm}{\textbf{Theorem}}[section]
\newtheorem{Lem}[Thm]{Lemma}
\newtheorem{Fac}[Thm]{Fact}
\newtheorem{Cor}[Thm]{Corollary}
\newtheorem{Prop}[Thm]{Proposition}
\newtheorem{Conj}[Thm]{Conjecture}

\theoremstyle{remark}
\newtheorem{Rem}[Thm]{Remark}

\theoremstyle{definition}
\newtheorem{Def}[Thm]{Definition}

\newtheorem*{ack}{Acknowledgments}
\newtheorem*{fund}{Funding}

\newcommand{\Aut}{\mathop{\mathrm{Aut}}\nolimits}

\newcommand{\GL}{\mathop{\mathrm{GL}}\nolimits}

\newcommand{\Proj}{\mathop{\mathrm{Proj}}\nolimits}
\newcommand{\Pic}{\mathop{\mathrm{Pic}}\nolimits}

\newcommand{\Supp}{\mathop{\mathrm{Supp}}\nolimits}

\newcommand{\mult}{\mathop{\mathrm{mult}}\nolimits}

\newcommand{\DF}{\mathop{\mathrm{DF}}\nolimits}

\excludeversion{Eliminates}

\excludeversion{N.B.}

\excludeversion{memo}

\begin{document}

\title{The Calabi conjecture and K-stability} 

\author{Yuji Odaka}
\address{Research Institute for Mathematical Sciences (RIMS), 
Kyoto University, Oiwake-cho, Kitashirakawa, Sakyo-ku, Kyoto
606-8502, Japan}
\email{yodaka@kurims.kyoto-u.ac.jp}

\begin{abstract}
We algebraically prove K-stability of polarized Calabi-Yau varieties and canonically polarized  varieties  with mild singularities. \textcolor{black}{In particular, the} ``stable varieties" introduced by Koll\'{a}r-Shepherd-Barron \cite{KSB88} and Alexeev \cite{Ale94}, which form compact moduli space, are proven to be K-stable although 
it is well known that they are \textit{not} necessarily asymptotically (semi)stable. As a consequence, we have \textit{orbifold} counterexamples, to 
the folklore conjecture ``K-stability implies asymptotic stability". 
They have K\"{a}hler-Einstein (orbifold) metrics so the result of 
Donaldson \cite{Don01} does \textit{not} hold for \textit{orbifolds}. 

\end{abstract}


\maketitle


\section{Introduction}
Throughout, we work over $\mathbb{C}$, the field of complex numbers. 
The original GIT stability notion for polarized variety is \textit{asymptotic \textcolor{black}{$($}Chow or Hilbert\textcolor{black}{$)$} stability} which was studied by Mumford, Gieseker etc (cf.\ \cite{Mum77}, \cite{Gie77}, \cite{Gie82} ). The newer version \textit{K-stability} of polarized variety is defined as positivity of \textcolor{black}{the} \textit{Donaldson-Futaki invariants}\footnote{It is also called the generalized Futaki invariants or 
simply called the Futaki invariants by S.~K.~Donaldson. } \cite{Don02}, 
a kind of GIT weights, which \textcolor{black}{is}
 a reformulation of Tian's original notion \cite{Tia97}. 
It is introduced \textcolor{black}{with an expectation to be the} 
\textcolor{black}{algebro-geometric} 
counterpart of the existence of K\"{a}hler-Einstein metrics or more generally K\"{a}hler metrics with constant scalar curvature (cscK). 

Let us recall that \textcolor{black}{the 
Donaldson-Futaki invariant is a rational number associated to a 
\textit{test configuration} (which correspond to $1$-parameter subgroup)} and it is just 
a ``leading coefficient" of the \textcolor{black}{sequence} 
of Chow weights with respect to twists \textcolor{black}{of the polarization} of the test configuration, while asymptotic Chow 
stability is, \textcolor{black}{roughly speaking, defined by} 
 ``all asymptotic behaviour" of \textcolor{black}{Chow weights rather than just by 
their leading coefficients. }
For the \textcolor{black}{details on these notions}, we refer to 
\cite[section2]{RT07}, \cite{Mab08a} and the review \cite[section 2]{Od09b}. 

In the previous paper \cite{Od09b}, we reformed a\textcolor{black}{n algebro-geometric}  formula of \textcolor{black}{the} 
Donaldson-Futaki invariants by X.~Wang \cite[Proposition 19]{Wan08}, 
 and gave its applications; we established K-(semi)stabilities \textcolor{black}{of} some classes of polarized varieties. This paper is a sequel to that paper. 

By $(X,L)$, we \begin{Eliminates}\textcolor{green}{usually}\end{Eliminates} denote 
\textcolor{black}{an} equidimensional polarized projective variety (i.e.\ reduced)\textcolor{black}{,} which is not necessarily smooth\textcolor{black}{,} 
with $\dim(X)=n$. 
\textcolor{black}{Moreover, w}e \textcolor{black}{always} assume that $X$ is $\mathbb{Q}$-Gorenstein, is Gorenstein in codimension $1$ and satisfies Serre condition $S_{2}$. 
These technical conditions are put to ensure that the canonical divisor $K_{X}$ or sheaf $\omega_{X}$ is in a tractable class (cf.\ e.g., \cite{Ale96}). 

\textcolor{black}{The following is the main result of this paper. }

\begin{Thm}[=Theorem \ref{moreKst.1} and \ref{moreKst.2}]\label{i.moreKst}

{\rm (i)}
A semi-log-canonical $($pluri$)$canonically polarized variety $(X, \mathcal{O}_{X}(mK_{X}))$, where $m \in \mathbb{Z}_{>0}$, is K-stable. 

{\rm (ii)}
A log-terminal polarized variety $(X,L)$ with numerically trivial canonical divisor 
$K_{X}$ is K-stable. 

\end{Thm}

\textcolor{black}{The semi-log-canonicity or the log terminality (which is stronger) is the mildness of singularities, and is defined in terms of \textit{discrepancy}, which is developed along 
the minimal model program (cf.\ e.g., \cite[section 2.3]{KM98}, \cite{Ale96}). 
For the general effects of 
singularities on stability, consult \cite{Od09a}. }
We have the following differential geometric 
background, originally known as the Calabi conjecture, which became a theorem 
more than thirty years ago. 

\begin{Fac}\label{fact:AY}

{\rm (i)}({\cite{Aub76}, \cite{Yau78}})
A smooth projective manifold $X$ with ample canonical divisor $K_{X}$ has a K\"ahler-Einstein metric. 

{\rm (ii)}({\cite{Yau78}})
A smooth polarized manifold $(X,L)$ with numerically trivial canonical divisor $K_{X}$ has a K\"ahler-Einstein metric with K\"ahler class 
$c_{1}(L)$. 
\end{Fac}
Let us recall the following conjecture on general existence of K\"ahler-Einstein 
metrics, which was recently formulated. 

\begin{Conj}[cf.\ \cite{Yau90}, \cite{Tia97}, \cite{Don02}]\label{DTY}
Let $(X,L)$ be a smooth polarized manifold with $c_{1}(X)=ac_{1}(L)$ with $a\in \mathbb{R}$. Then $X$ has a K\"ahler-Einstein metric with K\"ahler class $c_{1}(L)$ 
if and only if $(X,\mathcal{O}_{X}(-K_{X}))$ is 
K-polystable. 
\end{Conj} 
\noindent
From the recent progress in Conjecture \ref{DTY} (in particular, \cite{Tia97}, \cite{Don05}, \cite{CT08}, \cite{Stp09}, \cite{Mab08b} and \cite{Mab09}), one direction is proved as follows. 
\begin{Fac}\label{fact:YTD}
If a projective manifold $X$ admits a 
K\"ahler-Einstein metric with K\"ahler class $c_1(L)$, then $(X,L)$ is K-polystable. 
\end{Fac}
Combining Fact \ref{fact:AY} and \ref{fact:YTD}, we find that smooth 
canonically polarized manifolds and smooth polarized Calabi-Yau manifolds over 
$\mathbb{C}$ are K-polystable. 
In this point of view, the purpose of this paper is recovering this relation directly in  purely algebro-geometric way and to give some applications. 

\textcolor{black}{For the first application}, we note that K-stability of $(X,L)$ is slightly stronger than K-polystability (cf.\ e.\ g.\ \cite[section 3]{RT07}), in the sense that it requires that the automorphism group $\Aut(X,L)$ does not contain any subgroup isomorphic to $\mathbb{G}_{m}$. 
\textcolor{black}{Therefore}, combining Theorem \ref{i.moreKst} {\rm (ii)} with the theorem of Matsushima \cite{Mat57}, which works for orbifolds in general, 
we have the following corollaries. 

\begin{Cor}\label{aut:fin.CY}
Let $(X,L)$ be a polarized $($projective$)$ orbifold 
with numerically trivial canonical divisor $K_X$. Then, $\Aut(X,L)$ is a finite group. 
\end{Cor}

\begin{Cor}{\label{aut.CY}}
Let $X$ be a projective orbifold with numerically trivial canonical divisor $K_X$. 
Then, the connected component $\Aut^{0}(X)$ of the automorphism group $\Aut(X)$ 
is an abelian variety. 
\end{Cor}

\noindent
The author could not find any other proofs of these corollaries which work for 
orbifolds in literatures, although we can also partially prove it via a different approach 
in the case when $X$ has only canonical singularities. Indeed, after taking a $\Aut^{0}(X)$-equivariant resolution $\tilde{X}$ of $X$, these follows from arguments in  \cite[Chapter 11]{Iit82} which says the automorphism of the non-uniruled projective manifold $\tilde{X}$ should not  include $\mathbb{G}_{m}$ nor $\mathbb{G}_{a}$. 






Let us recall that the moduli of stable curves $\bar{M_{g}}$ is constructed in the geometric invariant theory. 
\textcolor{black}{As} higher dimension\textcolor{black}{al generalization}, 
it was recently proved that the \textit{stable varieties} admitting 
\textit{semi-log-canonical singularities} also forms projective moduli as well 
by using LMMP method, not relying on the GIT theory \textcolor{black}{(cf.\ e.\ g.\ \cite{KSB88}, \cite{Kol90}, \cite{Ale94}, \cite{Vie95}, \cite{AH09})}. 
Along the development of that generalization, a fundamental observation was that such a stable variety is \textit{not} 
necessarily asymptotically stable (cf.\ \cite{She83}, 
\cite{Kol90}, \cite[especially 1.7]{Ale94}). 

\textcolor{black}{As a second application of Theorem \ref{i.moreKst}, }
we will prove that \textcolor{black}{there are orbifolds counterexamples with discrete automophism groups}, to the folklore conjecture ``K-(poly)stability implies asymptotic (poly)stability". \textcolor{black}{On the other hand, it seems to be affirmatively proved for the case that the polarized variety is \textit{smooth} with \textit{discrete} automorphism group by Mabuchi and Nitta \cite{MN}. 
Therefore, the problem is quite subtle. }

\textcolor{black}{We should note that 
the first counterexample for \textit{non-discrete} automorphism group case} had been found as a smooth toric Fano $7$-fold by Ono-Sano-Yotsutani \cite{OSY09}. 
Recently, another example was found by Della Vedova and Zu\textcolor{black}{ddas}  \cite{DVZ10} 
which is a smooth rational projective surface whose automorphism group is also  \textit{not} discrete. 
The key \textcolor{black}{for our construction} is the theory on the effects of singularities 
on the \textit{asymptotic} (semi)stability by Eisenbud and Mumford \cite[section $2$]{Mum77}\textcolor{black}{; so-called ``local stability" theory. } \begin{Eliminates}\textcolor{green}{We should note that 
the first counterexample had been found as a smooth toric Fano $7$-fold by Ono-Sano-Yotsutani \cite{OSY09}.} \end{Eliminates} \textcolor{black}{Our} counterexamples also 
have K\"{a}hler-Einstein metrics.

\begin{Cor}[cf.\ Corollary \ref{counterexamples.1}]\label{i.counterexamples}

{\rm (i)}
There are projective orbifolds $X$ with ample canonical divisor\textcolor{black}{s} $K_{X}$ 
which have K\"{a}hler-Einstein $($orbifold$)$ metric\textcolor{black}{s}, 
and $(X, K_{X})$ are K-stable 
but asymtotically Chow unstable. 

{\rm (ii)}
There are polarized orbifolds $X$ with numerically trivial canonical divisor\textcolor{black}{s} $K_{X}$ and discrete automorphism group\textcolor{black}{s} $\Aut(X)$ 
such that for any polarization $L$, $X$ have Ricci-flat K\"{a}hler $($orbifold$)$ 
metric with K\"{a}hler class $c_{1}(L)$ and 
$(X,L)$ \textcolor{black}{are} K-stable but asymptotically Chow unstable. 

\end{Cor}

We will show \textcolor{black}{the} examples explicitly in section $3$. Since our examples have discrete automorphism groups, 
these are also examples which show that Donaldson's result \cite[Corollary 4]{Don01} 
does \textit{not} hold for orbifolds.

\section{K-stability of Calabi-Yau variety and of canonical model}

Firstly, let us recall the definition of K-stability. 
For that, we prepare the following concepts. 

\begin{Def}

A \textit{test configuration} (resp.\ \textit{semi test configuration}) for a polarized variety $(X,L)$ is a 
polarized variety $(\mathcal{X},\mathcal{L})$ with: 
\begin{enumerate}
\item{a $\mathbb{G}_{m}$ action on $(\mathcal{X},\mathcal{L})$}
\item{a proper flat morphism $\alpha\colon \mathcal{X} \rightarrow \mathbb{A}^{1}$}
\end{enumerate}
such that $\alpha$ is $\mathbb{G}_{m}$-equivariant for the usual action on $\mathbb{A}^{1}$: 
\begin{align*}
\mathbb{G}_{m}\times \mathbb{A}^{1}&& \longrightarrow&& \mathbb{A}^{1}\\
                          (t,x)    && \longmapsto    &&    tx,      
\end{align*}
$\mathcal{L}$ is relatively ample (resp.\ relatively semi ample), 
and $(\mathcal{X},\mathcal{L})|_{\alpha^{-1}(\mathbb{A}^{1}\setminus\{0\})}$ is $\mathbb{G}_{m}$-equivariantly isomorphic 
to $(X,L^{\otimes r})\times (\mathbb{A}^{1}\setminus\{0\})$ for some positive integer $r$, called the \textit{exponent}, 
with the natural action of $\mathbb{G}_{m}$ on the latter and the trivial action on the former. 

\end{Def}

\noindent
These concepts above gives a setting to \textit{geometrize} 
one parameter subgroup of general linear group of section of the line bundle, in the 
following sense. 

\begin{Prop}[{\cite[Proposition 3.7]{RT07}}]\label{tc.1-ps}

In the above situation, a one-parameter subgroup of $GL(H^{0}(X,L^{\otimes{r}}))$ is equivalent to the data of a test configuration $(\mathcal{X},\mathcal{L})$ 
with very ample polarization $\mathcal{L}$ and exponent $r$ of $(X,L)$ for $r \gg 0$. 

\end{Prop}

\noindent
In fact, let $\lambda \colon \mathbb{G}_{m}\rightarrow \GL(H^{0}(X,L^{\otimes r}))$ 
be a one parameter subgroup. Then, consider the natural action $\lambda \times \rho$  of $\mathbb{G}_{m}$ on $(\mathbb{P}(H^{0}(X,L^{\otimes r}))\times \mathbb{A}^{1},  \mathcal{O}(1))$ as a polarized variety, 
where  $\rho$ is the multiplication action on $\mathbb{A}^{1}$. 
Then the closure of the orbit $\mathcal{X}:=\overline{((\lambda \times \rho) (\mathbb{G}_{m}))(X \times \{1\})}$ is a test configuration with the natural polarization $\mathcal{O}(1)|_{\mathcal{X}}$ and the restriction of the natural action on 
$(\mathbb{P}(H^{0}(X,L^{\otimes r}))\times \mathbb{A}^{1},  \mathcal{O}(1))$. This is 
called the \textit{DeConcini-Procesi family} of $\lambda$ by Mabuchi in \cite{Mab08a}. The fact that any (very ample) test configuration can be obtained in this way follows from the $\mathbb{G}_{m}$-equivariant version of Serre's conjecture on the vector bundle on $\mathbb{A}^{1}$ (cf.\ \cite[Lemma 2]{Don05}). 

The Donaldson-Futaki invariant is a rational number 
associated to each semi test configuration. For a test configuration, 
it is roughly a sort of GIT weight, which is a leading coefficient of a sequence of Chow weights with respect to twist of the polarization of the test configuration. 

We explain the detailed definition of the Donaldson-Futaki invariants here. 
Let $P(k):=\dim H^0(X,L^{\otimes k})$, which is a polynomial in $k$ of degree $n$ due to  the Riemann-Roch theorem. 
Since the $\mathbb{G}_m$-action preserves the central fibre $(\mathcal{X}_0,\mathcal{L}|_{\mathcal{X}_{0}})$ of $\mathcal{X}$, $\mathbb{G}_m$ acts also on $H^0(\mathcal{X}_0,\mathcal{L}^{\otimes K}\mid_{\mathcal{X}_0})$, where $K\in \mathbb{Z}_{>0}$. 
Let $w(Kr)$ be the weight of the induced action on the highest exterior power of $H^0(\mathcal{X}_0,\mathcal{L}^{\otimes K}\mid_{\mathcal{X}_0})$, which is a polynomial of $K$ of degree $n+1$ due to the Mumford's droll Lemma (cf.\ \cite[Lemma 2.14]{Mum77} and \cite[Lemma 3.3]{Od09b}) and the Riemann-Roch theorem.
Here, the \textit{total weight} of an action of $\mathbb{G}_{m}$ on some finite-dimensional vector space is defined as the sum of all weights, where the \textit{weights} mean the exponents of eigenvalues which should be powers of $t\in \mathbb{A}^1$. 
Let us denote the projection from $\mathcal{X}$ to $\mathbb{A}^{1}$ by $\Pi$. 
Let us take $rP(r)$-th power and 
SL-normalize the action of $\mathbb{G}_{m}$ on 
$(\Pi_{*}\mathcal{L})|_{\{0\}}$, then the corresponding normalized weight on 
$(\Pi_{*}\mathcal{L}^{\otimes K})|_{\{0\}}$ 
is $\tilde{w}_{r,Kr}:=w(k)rP(r)-w(r)kP(k)$, where $k:=Kr$. It is a 
polynomial of form 
$\sum_{i=0}^{n+1}e_{i}(r)k^{i}$ of degree $n+1$ in $k$ for $k \gg 0$, with coefficients which are also 
polynomial 
of degree $n+1$ in $r$ 
for $r \gg 0$ : $e_{i}(r)=\sum_{j=0}^{n+1}e_{i,j}r^{j}$ for $r \gg 0$. 
Since the weight is normalized, $e_{n+1,n+1}=0$. 
The coefficient $e_{n+1,n}$ is 
called the \textit{Donaldson-Futaki invariant} of the test configuration, which we denote by $\DF(\mathcal{X},\mathcal{L})$. 
For an arbitrary \textit{semi} test configuration $(\mathcal{X},\mathcal{L})$ 
of order $r$, we can also define the Donaldson-Futaki invariant as well by setting $w(Kr)$ as the total weight of the  induced action on $H^{0}(\mathcal{X}, \mathcal{L}^{\otimes K})/tH^{0}(\mathcal{X}, \mathcal{L}^{\otimes K})$ (cf. \cite{RT07}). 

Now, we can define K-stability and its versions as follows. 

\begin{Def}
We say that $(X,L)$ is \textit{K-stable} (resp. \textit{K-semistable}) 
if and only if $\DF(\mathcal{X},\mathcal{L})>0$ (resp. $\DF(\mathcal{X},\mathcal{L})\ge 0$) for any non-trivial test configuration. 
On the other hand, \textit{K-polystability} of $(X,L)$ means that $\DF\ge 0$ for any non-trivial test configuration and $\DF(\mathcal{X},\mathcal{L})=0$ only if a test configuration $(\mathcal{X},\mathcal{L})$ is a product test configuration. 
In particular, K-stability implies K-polystability and K-polystability implies 
K-semistability.  
\end{Def}


Let us recall that there are \textcolor{black}{algebro-geometric} formulae 
for the Donaldson-Futaki invariants, which was obtained in \cite{Wan08} and \cite{Od09b}. 

\begin{Thm}\label{DF.formula}

{\rm (i)}$($\cite[Proposition 19]{Wan08}$)$
Let $(\mathcal{X},\mathcal{M})$ be an $($ample$)$ test configuration 
of a polarized variety $(X,L)$, and let us denote its natural compactification 
as $(\bar{\mathcal{X}},  \bar{\mathcal{M}})$, a polarized family over $\mathbb{P}^{1}$ which is trivial (product) over $\mathbb{P}^{1}\setminus \{0 \}$. 
Then, the corresponding Donaldson-Futaki invariant is the following $;$ 

\begin{equation*}
\DF (\mathcal{X}, \mathcal{M})=\dfrac{1}{2(n!)((n+1)!)}
\bigl\{-n(L^{n-1}.K_{X})(\bar{\mathcal{M}}^{n+1})+(n+1)(L^{n})(\bar{\mathcal{M}}^{n}.K_{\bar{\mathcal{X}}/\textcolor{black}{\mathbb{P}^{1}}})\bigr\}. 
\end{equation*}

Here, $K_{\bar{\mathcal{X}}/\mathbb{P}^{1}}$ means the divisor $K_{\bar{\mathcal{X}}}-f^{*}K_{\mathbb{P}^{1}}$ with the projection $f \colon \bar{\mathcal{X}}\rightarrow \mathbb{P}^{1}$. 

{\rm (ii)}$($\cite[Theorem $3.2$]{Od09b}$)$
For any flag ideal $\mathcal{J}\subset \mathcal{O}_{X\times \mathbb{A}^{1}}$ 
$($cf.\ \cite[Definitio\textcolor{black}{n} 3.1]{Od09b}$)$, 
consider the ``semi" test configuration $(Bl_{\mathcal{J}}(X\times \mathbb{A}^{1})=:\mathcal{B},\mathcal{L}(-E))$ of blow up type 
with $($relatively$)$ ``semi"ample $\mathcal{L}(-E)$ where $\Pi^{-1}\mathcal{J}=\mathcal{O}_{\mathcal{B}}(-E)$. 
Here, $\Pi \colon \mathcal{B} \rightarrow X\times \mathbb{A}^{1}$ 
\textcolor{black}{is} the blowing up morphism. 
Let us write its natural compactification \textcolor{black}{as} 
$(Bl_{\mathcal{J}}(X\times \mathbb{P}^{1})=:\bar{\mathcal{B}},\overline{\mathcal{L}(-E)})$, which is obtained by blowing up the same flag ideal $\mathcal{J}$ on 
$X\times \mathbb{P}^{1}$ and \textcolor{black}{let} $p_{i}$ $(i=1,2)$ be the projection from $X\times\mathbb{P}^{1}$. 
Then, \textcolor{black}{if $\mathcal{B}$ is Gorenstein in codimension $1$, }
the Donaldson-Futaki invariant of the semi test configuration can be expanded in
 the following way$;$ 
\begin{equation*}
 \begin{split}
   & 2(n!)((n+1)!)\DF (\mathcal{B},\mathcal{L}(-E)) \\
  &=-n(L^{n-1}.K_{X})(\overline{(\mathcal{L}(-E))}^{n+1})+(n+1)(L^{n})(\overline{(\mathcal{L}(-E))}^{n}.K_{\bar{\mathcal{B}}/\textcolor{black}{\mathbb{P}^{1}}}) \\
  &=-n(L^{n-1}.K_{X})(\overline{(\mathcal{L}(-E))}^{n+1})+(n+1)(L^{n})(\overline{(\mathcal{L}(-E))}^{n}.\Pi^{*}(p_{1}^{*}K_{X})) \\
   & +(n+1)(L^{n})(\overline{(\mathcal{L}(-E))}^{n}.K_{\bar{\mathcal{B}}/X\times \mathbb{P}^{1}}). 
 \end{split}
\end{equation*}
Here, $K_{\bar{\mathcal{B}}/X\times \mathbb{P}^{1}}$ means $K_{\bar{\mathcal{B}}}-\Pi^{*}K_{X\times \mathbb{P}^{1}}$. 
\end{Thm}

The \textit{flag ideal} $\mathcal{J}\subset 
\mathcal{O}_{X\times \mathbb{A}^{1}}$ \textcolor{black}{means} 
a coherent ideal of the form 
\begin{equation*}
\mathcal{J}=I_{0}+I_{1}t+I_{2}t^{2}+\cdots +I_{N-1}t^{N-1}+(t^N), 
\end{equation*}
where $I_{0} \subset I_{1} \subset \cdots I_{N-1} \subset \mathcal{O}_{X}$ 
is a sequence of coherent ideals \textcolor{black}{of $X$}
(cf.\ \cite[Definitio\textcolor{black}{n} 3.1]{Od09b}). 
\textcolor{black}{The formula} ${\rm (ii)}$ is useful by its form. 
The canonical divisor part is defined as 
\begin{equation*}
\DF_{cdp}(\mathcal{B}, \mathcal{L}(-E))=
-n(L^{n-1}.K_{X})(\overline{(\mathcal{L}(-E))}^{n+1})
+(n+1)(L^{n})(\overline{(\mathcal{L}(-E))}^{n}.\Pi^{*}(p_{1}^{*}K_{X})) 
\end{equation*} and the discrepancy term is defined as 
\begin{equation*}
\DF_{dt}(\mathcal{B}, \mathcal{L}(-E)):=(n+1)(L^{n})(\overline{(\mathcal{L}(-E))}^{n}.K_{\bar{\mathcal{B}}/X\times \mathbb{P}^{1}}). 
\end{equation*}

Obviously, the Donaldson-Futaki invariant is the sum of the canonical divisor part and 
the discrepancy term, up to a positive constant. 
Roughly speaking, the canonical divisor part reflects the positivity of the 
canonical divisor and the discrepancy term reflects the mildness of singularity. 
Consult \cite{Od09b} for the detail. 
In this paper, we use the formula {\rm (ii)} for applications. 
A key for our applications of ${\rm (ii)}$ is that we allow ``semi" test configurations, not only genuine (ample) test configurations, so that the following holds. 

\begin{Prop}[{\cite[Proposition 3.10 $(\rm {ii})$]{Od09b}}]\label{formula.enough}

$(X,L)$ is K-stable \textcolor{black}{if and only if}\color{black}{} 
for all \textcolor{black}{``semi"}
test configurations of the type \ref{DF.formula} ${\rm (ii)}$ $($i.e.\  $(\mathcal{B}=Bl_{\mathcal{J}}(X\times \mathbb{A}^{1}), \mathcal{L}^{\otimes{r}}(-E))$ $)$ with $\mathcal{B}$ Gorenstein in codimension $1$, the Donaldson-Futaki invariant is positive. 

\end{Prop}

\noindent
We roughly explain the proof of Proposition \ref{formula.enough}. 
The proof is based on the fact that any non-trivial test configuration is 
birationally dominated by a semi test configuration of the blow up type as above, 
and \cite[Proposition 5.1]{RT07} shows that the dominating semi test configuration 
should have an equal or less Donaldson-Futaki invariant than the Donaldson-Futaki invariant of the original test configuration. Therefore, the if part holds, which is sufficient  for our applications in this paper. Moreover, taking the dominating semi test configuration  carefully after the argument of \cite[section 2]{Mum77}, we can see that those 
two Donaldson-Futaki invariants are actually the same since the global section of 
the twisted polarization of those (semi) test configurations are the same. 
Furthermore, for any semi test configuration $(\mathcal{Y},\mathcal{M})$, 
if we take sufficiently divisible positive integer $c$, we can birationally contract 
$(\mathcal{Y},\mathcal{M}^{\otimes c})$ to get an (ample) test configuration $(\Proj(\oplus_{k\geq 0}H^{0}(\mathcal{Y},
\mathcal{M}^{\otimes k})),\mathcal{O}(c))$ with the same Donaldson-Futaki invariant. 
Therefore, the only if part also holds. 

\textcolor{black}
{Now, let us prove the first main theorem. }

\begin{Thm}\label{moreKst.1}
A semi-log-canonical \textcolor{black}{$($}pluri\textcolor{black}{$)$}canonically polarized variety $(X,\mathcal{O}_{X}(mK_{X}))$\textcolor{black}{, where $m \in \mathbb{Z}_{>0}$,} is K-stable. 
\end{Thm}

\begin{proof}[Proof of Theorem \ref{moreKst.1}]
We use the formula \ref{DF.formula} $(\rm {ii})$. 
The canonical divisor part of Donaldson-Futaki invariant for $(\mathcal{B}, \mathcal{L}(-E))$ is $m^{n-1}(K_{X}^{n})(\overline{(\mathcal{L}(-E))}^{n}.(\overline{\mathcal{L}(nE)}))$. 
On the other hand, the discrepancy term is nonnegative by semi-log-canonicity (cf.\ \cite[proof of the ``only if part" of Proposition(5.5)]{Od09b}). 
Therefore, it is enough to prove that the canonical divisor part is strictly positive. We note that $\overline{\mathcal{L}(-E)}$ is not necessarily nef, as 
$\overline{(\mathcal{L}(-E))}^{n+1}=(-E)^{n+1}<0$ 
in the case when $\Supp(\mathcal{O}/\mathcal{J})$ has zero dimension. 
We prepare the following elementary Lemma. 

\begin{Lem}\label{elem.pol}

{\rm (i)}We have the following equality of polynomials with two variables; 
\begin{equation*}
(x-y)^{n}(x+ny)=x^{n+1}-\sum_{i=1}^{n}(n+1-i)(x-y)^{n-i}x^{i-1}y^{2}. 
\end{equation*}

{\rm (ii)}
The polynomials $(x-y)^{n-i}x^{i-1}y^{2}$ 
for $1 \leq i \leq n$ are linearly independent 
\textcolor{black}{over $\mathbb{Q}$} and the monomial $x^{s}y^{n+1-s}$ 
can be written as a linear combination of these 
\textcolor{black}{with integer coefficients}, for an arbitrary $s$ with $0<s<n$. 

\end{Lem}

We omit the proof of Lemma \ref{elem.pol}, since it is easy and given by 
simple calculation. 
By using Lemma \ref{elem.pol}, we can decompose the canonical divisor part of the Donaldson-Futaki invariants of 
$(\mathcal{B},\mathcal{L}(-E))$ as follows. 
We note that $(\overline{\mathcal{L}}^{n+1})=0$. 

\begin{equation}\label{DF1}
\DF_{cdp}(\mathcal{B},\mathcal{L}(-E))=m^{n-1}(K_{X}^{n})\bigl\{(-E^{2}.\sum_{i=1}^{n}(n+1-i)\overline{(\mathcal{L}(-E))}^{n-i}.\overline{\mathcal{L}}^{i-1})\bigr\},  
\end{equation}
where $s=\dim(\Supp(\mathcal{O}/\mathcal{J}))$.  

If $s<n$, then the description (\ref{DF1}) can be modified to the following form, 
thanks to Lemma \ref{elem.pol} (\rm ii). 
\begin{equation}\label{DF2}
m^{n-1}(K_{X}^{n})\bigl\{(-E^{2}.\sum_{i=1}^{n}(n+1-i+\epsilon_{i})\bigl(\overline{(\mathcal{L}(-E))}^
{n-i}. \overline{\mathcal{L}}^
{i-1}\bigr)-\epsilon '((-E)^{n+1-s}.\overline{\mathcal{L}}^{s})\bigr\}.  
\end{equation}
\noindent
Here, $\epsilon_{i}(1\geq i \geq n)$ and $\epsilon '$ are real numbers such that $0<|\epsilon_{i}|\ll 1$ and $0<\epsilon ' \ll 1$. 
And we have the following inequalities for each terms. 
\begin{Lem}\label{piece.ineq}

{\rm (i)}
$(-E^{2}.\overline{(\mathcal{L}(-E))}^{n-i}.\bar{\mathcal{L}}^{i-1})\geq 0$ for any $0<i<n$. 

{\rm (ii)}
$((-E)^{n\textcolor{black}{+}1-s}.\bar{\mathcal{L}}^{s})<0$ if $s<n$. 

{\rm (iii)}
$(-E^{2}.\mathcal{L}^{n-1})>0$ if $s=n$. 
\end{Lem}
\begin{proof}[Proof of Lemma \ref{piece.ineq}] 
Let us take a general member of $|lL|$ for $l \gg 0$, which we denote $H$. 
By cutting $X\times \mathbb{P}^{1}$ by $H \times \mathbb{P}^{1}$ 
and repeat it several times, 
we can reduce the proof to the case $i=1$ for $\rm (i)$, 
to the case with $s=0$ for $\rm (ii)$, and to the case when 
$X$ is a nodal curve for ({\rm iii}). 

Then, $\rm (i)$ follows from the Hodge index theorem 
and $\rm (ii)$ follows from the relative ampleness of $(-E)$. 

For $\rm (iii)$, we can assume without loss of generality that $0 \neq I_{0}$ 
(recall that $\mathcal{J}=\sum I_{i}t^{i}$), or in other words, $\mathcal{O}/\mathcal{J}$ is supported on proper closed subset of $X\times \{0\}$. If it is not the case, we can divide $\mathcal{J}$ by some power of $t$ without changing the Donaldson-Futaki invariant. Moreover, we can assume that $X$ is 
smooth by considering the normalization of $X\times \mathbb{P}^{1}$ and the 
pullback of the flag ideal $\mathcal{J}$ to it instead. 
In that case, $\mathcal{O}/\mathcal{J}$ is supported on finite points 
on at least one connected component, since we assumed $0\neq I_{0}$. 
We have $(-E|_{S})^{2}>0$ by relatively ampleness of $-E$ and 
$(-E|_{(X\times \mathbb{P}^{1})\setminus S})^{2}\geq 0$ by the Hodge index theorem. 
This completes the proof of $\rm (iii)$.

\end{proof}
Therefore, $\DF(\mathcal{B}, \mathcal{L}(-E))>0$ follows from 
Lemma \ref{piece.ineq} (i)(ii) for the case with $s<n$, due to the description of 
the Donaldson-Futaki invariant (\ref{DF2}). If $s=n$, 
then $\DF(\mathcal{B}, \mathcal{L}(-E))>0$ follows from 
Lemma \ref{piece.ineq} (i)(iii) and the description of the Donaldson-Futaki invariant  (\ref{DF1}). 
\end{proof}

\begin{Rem}\label{aut:fin.gen.type}
From Theorem \ref{moreKst.1}, the automorphism group $\Aut(X)$ for an arbitrary semi log canonical projective variety $X$ with ample canonical $\mathbb{Q}$-Cartier  divisor $K_{X}$ has no 
nontrivial reductive subgroup. 
Let us recall that it is furthermore a common knowledge that $\Aut(X)$ is 
actually finite for such $X$. 
\textcolor{black}{C}onsult Iitaka's book \cite[Theorem(10.11) and Theorem(11.12)]{Iit82} for the usual proof. 
But it is impressive \textcolor{black}{to the author} 
that these calculation of the Donaldson-Futaki invariants derives such a nontrivial result on $\Aut(X)$, which is a quite different from the usual approach. 

\end{Rem}

\textcolor{black}{Let us proceed to the second main theorem. }

\begin{Thm}\label{moreKst.2}

A log-terminal polarized variety $(X,L)$ with numerically trivial 
canonical divisor $K_{X}$ 
is K-stable. 

\end{Thm}

This theorem with the theorem of Matsushima \cite{Mat57} yields the 
following colloraries. 

\begin{Cor}\label{aut:fin}
Let $(X,L)$ be a polarized $($projective$)$ orbifold 
with numerically trivial canonical 
divisor $K_X$. Then, $\Aut(X,L)$ is a finite group. 
\end{Cor}

\begin{Cor}\label{aut:av}
\textcolor{black}{Let $X$ be a projective orbifold with numerically trivial canonical divisor $K_X$. 
Then, the connected component $\Aut^{0}(X)$ of the automorphism group $\Aut(X)$ 
is an abelian variety. }
\end{Cor}

\noindent
We explained the proof of Colloraries \ref{aut:fin} in the introduction. 
Given Corollary \ref{aut:fin}, Corollary \ref{aut:av} can be proved as follows. 
Let us recall that $\Aut(X,L)$ is the isotropy subgroup of $\Aut(X)$ for 
the natural action of $\Aut(X)$ on $[L]\in \Pic(X)$ by the definition. 
We recall that the Picard scheme $\Pic^{0}(X)$ of $X$ is 
an abelian variety since $X$ is normal. 
Let $Z$ be a connected component of $\Pic(X)$, which includes $[L]$. Due to Corollary 
\ref{aut:fin}, the restriction of the translation morphism of $[L]$, $\Aut^{0}(X)\to 
Z$ is generically finite (onto the image). 
Since $Z$ is abelian variety, $\Aut^{0}(X)$ should not 
include proper linear algebraic subgroup which is rational. 
Therefore, $\Aut^{0}(X)$ should be 
an abelian variety by the theory of the Chevalley decomposition.

\begin{proof}[Proof of Theorem \ref{moreKst.2}]
From the formula of Donaldson-Futaki invariants \ref{DF.formula} {\rm (ii)} and 
Proposition \ref{formula.enough}, it is enough to prove that 
\[ (\overline{(\mathcal{L}(-E))}^{n}.K_{\bar{\mathcal{B}}/X\times \mathbb{P}^{1}}) \]
is positive. Since $X$ is assumed to be log-terminal, 
$(X\times \mathbb{A}^{1},X\times \{0\})$ is also (purely) log-terminal by the 
inversion of adjunction, which can be proved by considering the resolution of 
$X\times \mathbb{A}^{1}$ of the form $W\times \mathbb{A}^{1}$. 
Therefore, 
any coefficient of  $K_{\mathcal{B}/X\times \mathbb{P}^{1}}$ for exceptional prime divisor is positive.  
On the other hand, $\mathcal{L}(-E)$ is (relatively) semiample (over $\mathbb{A}^{1}$) 
on $\mathcal{B}$, so we have non-negativity of the term. 

Furthermore, since $K_{\mathcal{B}/X\times\mathbb{A}^{1}}-cE$ is effective for $0<c\ll 1$, it is enough to prove 
\begin{equation}\label{CYineq}
 (\overline{(\mathcal{L}(-E))}^{n}.E)>0. 
\end{equation}
Here, we have 
\[ \overline{((\mathcal{L}(-E))}^{n+1})=\overline{(\mathcal{L}(-E))}^{n+1}-(\bar{\mathcal{L}})^{n+1}=
(-E.\sum_{i=0}^{n}(\overline{(\mathcal{L}(-E))}^{i}.\bar{\mathcal{L}}^{n-i})\leq 0 \] 
and on the other hand, 
\[ (\overline{(\mathcal{L}(-E))}^{n}.(\overline{\mathcal{L}(nE)}))>0 \] 
from the proof of Theorem \ref{moreKst.1} and these implies (\ref{CYineq}). 
This ends the proof of Theorem \ref{moreKst.2}. 

\end{proof}

As a final remark in this section, we recall that the \textit{asymptotic} stability of these polarized variety for smooth case is already known by a simple combination of the results of  \cite{Aub76}, \cite{Yau78} and \cite[Corollary 4]{Don01} via 
\textcolor{black}{the existence of K\"ahler-Einstein metrics. }
We note that we can apply \cite[Corollary 4]{Don01} thanks to 
the discreteness of $\Aut(X,L)$ (see Corollary \ref{aut:fin.CY} and 
\cite[Theorem(10.11), Theorem(11.12)]{Iit82}). 

\begin{Prop}[cf.\ \cite{Aub76}, \cite{Yau78}, \cite{Don01}]
{\rm (i)}
A smooth $($pluri$)$canonically polarized manifold $(X, \mathcal{O}_{X}(mK_{X}))$ over 
$\mathbb{C}$, where $m \in \mathbb{Z}_{>0}$, is asymptotically stable. 

{\rm (ii)}
A smooth polarized manifold $(X,L)$ with numerically trivial canonical divisor 
$K_{X}$ is asymptotically stable. 
\end{Prop}








\section{\textcolor{black}{K\"ahler-Einstein,} K-stable but asymptotically unstable orbifolds}

Let us recall the asymptotic stabilities. These notions are the original GIT stability notions for polarized varieties. 

\begin{Def}

A polarized scheme $(X,L)$ is said to be \textit{asymptotically Chow stable} (resp.\ \textit{asymptotically Hilbert stable}, 
\textit{asymptotically Chow semistable}, 
\textit{asymptotically Hilbert semistable}), if for an arbitrary 
$m \gg 0$, $\phi _{m}(X)\subset \mathbb{P}(H^{0}(X, L^{\otimes{m}}))$ is Chow stable (resp.\ Hilbert stable, Chow semistable, 
Hilbert semistable), where $\phi _{m}$ 
is the closed immersion defined by the complete linear system $|L^{\otimes{m}}|$. 
 
\end{Def} 
\noindent
As the Chow-stability (resp.\ Hilbert stability) is a bona fide GIT stability notion, 
we can see it via GIT weights by the Hilbert-Mumford's numerical criterion, which we call  the Chow weights (resp.\ Hilbert weights). 

On the other hand, the Donaldson-Futaki invariant is a limit of Chow weights with respect to the twist of polarization of test configuration. Hence, it has been a natural conjecture of folklore status that K-(poly)stability \textcolor{black}{implies} asymptotic Chow (poly)stability. \textcolor{black}{
In fact, it seems to be affirmatively proved 
recently for the case when the polarized variety is \textit{smooth} 
with a \textit{discrete} automorphism group, by Mabuchi and Nitta \cite{MN}. }

If we admit \textit{non-discrete} automorphism groups, it does not hold \textit{in general} by Ono-Sano-Yotsutani \cite{OSY09}. They showed that an example of toric K\"{a}hler-Einstein manifold constructed in \cite{NP09}, which is non-symmetric in the sense of Batyrev-Selivanova \cite{BS99}, is just a counterexample. It is a smooth toric Fano $7$-fold with $12$ vertices in the Fano polytope and $64$ vertices in the moment polytope. Della Vedova and Zu\textcolor{black}{dd}as \cite[Proposition 1.4]{DVZ10} gave  another counterexample which is the projective plane blown up at four points of which all but one are aligned. It also has a non-discrete automorphism group. 

Here, we give other counterexamples 
\textcolor{black}{with \textit{discrete} automorphism groups, 
but admit quotient singularities. }


The following is the key to prove asymptotic unstability for our examples, which follows \textcolor{black}{from} Eisenbud-Mumford's 
\textit{local stability} theory in \cite[section $3$]{Mum77}. 

\begin{Prop}[{\cite[Proposition 3.12]{Mum77}}]\label{locst} 
For asymptotically Chow semistable polarized variety $(X,L)$, 
$\mult(x,X)\leq (\dim X+1)!$ for any closed point $x \in X$. 
\end{Prop}

Combining with our Theorem \ref{moreKst.1} and Theorem \ref{moreKst.2}, we obtain the following. 

\begin{Cor}\label{counterexamples.1}

{\rm (i)}
For the following projective orbifolds $X$ which have discrete automorphism groups, 
$(X, K_{X})$ are K-stable 
but asymtotically Chow unstable. 
Furthermore,  they all have 
K\"{a}hler-Einstein $($orbifold$)$ metrics. 

{\rm (i-a)}
Finite quotients  of the selfproduct of Hurwitz curve $C$ 
$($e.g.\ , Klein curve $(x^{3}y+y^{3}z+z^{3}x=0) \subset \mathbb{P}^{2}$ with genus $3$ $)$ $X=(C\times C)/\Delta(\Aut(C))$. 
Here, $\Delta(\Aut(C))$ is the diagonal subgroup of $\Aut(C) \times \Aut(C)$. 

Here, a ``Hurwitz curve" means a smooth projective curve with $\#\Aut(C)=84(g-1)$, 
which is the maximum possible for the fixed genus $g(\ge 2)$ $($cf.\ \cite[section 6.10]{Iit82}$)$. 

{\rm (i-b)}
A quasi-smooth weighted projective hypersurface of the following type $;$ 

\begin{equation*}
(y^{p}x_{0}=\sum_{i=0}^{n}x_{i}^{c_{i}})\subset \mathbb{P}(a_{0},\cdots, a_{n},b), 
\end{equation*}

where $a_{i}c_{i}=pb+a_{0}$ and $p, c_{i} \gg 0$. 
It has $\frac{1}{b}(a_{1},\cdots ,a_{n})$-type cyclic quotient singularity, 
which has multiplicity bigger than $(n+2)!$, 
and the canonical divisor $K_{X}$ is ample $\mathbb{Q}$-Cartier divisor. 

{\rm (i-c)}
Let $l_{i}$ $($$i=1, \cdots, n$, where $n \geq 9$$)$ 
be general $n$ lines in projective plane $\mathbb{P}^{2}$. 
After the blowing up $\pi \colon B \rightarrow \mathbb{P}^{2}$ of $\cup (l_{i} \cap 
l_{j})$, let us blow down $\cup (\pi^{-1}_{*} l_{i})$ to obtain $X$. $X$ has cyclic quotient singularities with multiplicity $n-2$. $X$ is smoothable but not 
$\mathbb{Q}$-Gorenstein smoothable $($cf.\ \cite[section 2]{LP07}$)$. See also \cite{Kol08} and \cite{HK10} for similar  examples. 

{\rm (ii)}

For the following log Enriques surfaces $($cf.\ \cite{Zha91}, \cite{OZ00}$)$, 
for any polarization $L$, the polarized variety $(X,L)$ \textcolor{black}{are} K-stable but asymptotically Chow unstable. Furthermore, $X$ have Ricci-flat 
$($orbifold$)$ K\"{a}hler metrics with K\"{a}hler class $c_{1}(L)$. 

{\rm (ii-a)}
$X$=$Y/ \langle \sigma \rangle$, where $(Y,\sigma)$ is a K3 surface $Y$ with a non-symplectic automorphism $\sigma$ of finite order, 
in the list of \cite[Table$6$ l$1$ or Table$7$ l$1$]{AST09}. They have quotient singularity with multiplicity $17$ and $7$ 
respectively. 

{\rm (ii-b)}
$X$=$Z/ \langle \sigma \rangle$, where $Z$ is the birational crepant contraction of K3 surface $Y$ along a $(-2)$ curve $D$ on it, 
where $\sigma$ is a non-symplectic 
automorphism of finite order which fixes $D$, in the list of 
\cite[Table$3$ l$1$, Table$5$ l$1$]{AST09}. They have a quotient singularity with multiplicity $7$. 

\end{Cor}

\begin{proof}[Proof of Corollary \ref{counterexamples.1}]
These examples are asymptotically unstable by Proposition \ref{locst} and they have  K\"{a}hler-Einstein orbifold metrics 
by Yau \cite{Yau78} whose proof also works in the category of orbifolds. 
Alternatively, those examples that are (globally) finite quotients of smooth  projective varieties so we can also directly construct the metrics by descending from the covers. This is possible since the K\"{a}hler-Einstein metrics are unique up to $\Aut^{\circ}(X)$\textcolor{black}{, the connected component of $\Aut(X)$,} by 
Bando-Mabuchi \cite{BM87}. 
We proved the K-stability of examples {\rm (i)} in Theorem \ref{moreKst.1} and that of 
examples {\rm (ii)} in Theorem \ref{moreKst.2}. 
\end{proof}

\begin{Rem}
We are more examples of type ({\rm i}), of which we will omit the detail. 
They are $X$'s in \cite{LP07}, \cite{PPS09a}, \cite{PPS09b}. Consult those papers for the detail. They are ``$\mathbb{Q}$-Gorenstein\textcolor{black}{-}smoothable" rational projective surfaces and have ample $\mathbb{Q}$-Cartier canonical divisor $K_{X}$. 
For the concept of ``$\mathbb{Q}$-Gorenstein-smoothing", 
we refer to e.\ g.\ \cite[section $2$]{LP07} as well. 
They have quotient singularities with multiplicity larger than $6$. 
Consult also Rasdeaconu-Suvaina \cite{RS08} especially for the proof of ampleness of $K_{X}$ by explicit calculation of intersection numbers. 

The examples in (ii) are ``log Enriques surface"s, which are introduced by D.~Q.~Zhang in \cite{Zha91}. Original motivation of \cite{LP07}, \cite{PPS09a}, \cite{PPS09b} are 
to construct their smoothed deformation which are simply connected and $p_{g}=0$. 
\end{Rem}






\begin{fund}

This work was supported by the Grant-in-Aid for Scientific Research (KAKENHI No.\ 21-3748) and the Grant-in-Aid for JSPS fellows. 

\end{fund}

\begin{ack}
Firstly, the author would like to express his deep gratitudes to his advisor  Professor Shigefumi Mori for heartful encouragements, 
mathematical suggestions and reading the draft. 
I am also grateful to Doctor Shingo Taki, 
Professor Yongnam Lee and Professor Julius Ross for their helps and discussions. 
\end{ack}

\end{document}